\documentclass[11pt,a4paper,reqno]{article}
\usepackage{amsmath}
\usepackage{amsthm}
\usepackage{enumerate}
\usepackage{amssymb}

\usepackage{verbatim}
\usepackage{url}
\usepackage[latin1]{inputenc}

\usepackage{caption}
\usepackage{graphicx,color}
\def\be{\begin{equation}}
\def\ee{\end{equation}}
\def\bea{\begin{equation*}}
\def\eea{\end{equation*}}
\def\begs{\begin{split}}
\def\ends{\end{split}}

\newtheorem{thm}{Theorem}%[section]
\newtheorem{lma}[thm]{Lemma}

\theoremstyle{remark}
\newtheorem{preremark}[thm]{Remark}
\newtheorem{preex}[thm]{Example}

\begin{document}

\title{Coarsening in 2D slabs}
\date{\today}
\author{\Large Michael Damron \thanks{The research of M. D. is supported by NSF grants DMS-0901534 and DMS-1007626.}\\ \small Princeton University \\ \small Princeton, NJ, USA \and \Large Hana Kogan \thanks{The research of H. K. is supported by NSF grant OISE-0730136.}\\ \small Courant Institute, NYU \\ \small New York, NY, USA \and \Large Charles M. Newman \thanks{The research of C. M. N. is supported by NSF grants OISE-0730136, DMS-1007524 and DMS-1007626}\\ \small Courant Institute, NYU \\ \small New York, NY, USA \and \Large Vladas Sidoravicius \thanks{The research of V. S. is supported by CNPq grants PQ 308787/2011-0 and 484801/2011-2.} \\ \small IMPA \\ \small Rio de Janeiro, RJ, Brazil}
\maketitle

\begin{abstract}
We study coarsening; that is, the zero-temperature limit of Glauber dynamics in the standard Ising model on slabs $S_k = \mathbb{Z}^2 \times \{0, \ldots, k-1\}$ of all thicknesses $k \geq 2$ (with free and periodic boundary conditions in the third coordinate). We show that with free boundary conditions, for $k\geq 3$, some sites fixate for large times and some do not, whereas for $k=2$, all sites fixate. With periodic boundary conditions, for $k \geq 4$, some sites fixate and others do not, while for $k=2$ and $3$, all sites fixate. 
\end{abstract}

\section{Introduction}

Coarsening models have been extensively studied in the Physics literature -- see, for example, \cite[Ch. 9]{KRB10} and the references therein. These are stochastic Ising models at some low temperature $T_1$ whose initial state is chosen from the equilibrium distribution at a higher temperature $T_2$. The special case of the coarsening model we consider here is the case where $T_1 = 0$ and $T_2 = \infty$. The states are assignments of $\pm 1$ to the vertices of some graph and the most commonly studied graph is $\mathbb{Z}^d$ (with nearest neighbor edges) or finite box approximations to $\mathbb{Z}^d$ (with, for example, free or periodic boundary conditions). 

For $d=1$, the dynamics is exactly that of the standard voter model and it is an old result \cite{A83} that almost surely every site flips (between $+1$ and $-1$) infinitely often. For $d=2$, it was shown in \cite{NNS00}, that still every site flips infinitely often, but it is an open problem to determine what happens for $d \geq 3$. In \cite{NS00}, it was proposed, based on numerical results in the related issue of ``persistence" (sites which do not flip for a long time) from \cite{S94} that the flipping results for $d=1,2$ might change by dimension $4$ or $5$. But in fact, the situation is unclear even in dimension $3$. In this paper, in a first attempt to shed some light on the possible difference between $\mathbb{Z}^2$ and $\mathbb{Z}^3$ we study coarsening in slabs of varying thickness $k$ so as to interpolate between the full two and three dimensional lattices. To our surprise, there is more interesting structure in this $k$-dependence than we originally suspected.

\subsection{The model and definitions}
The slab $S_k, ~k \geq 2$, is the graph with vertex set $\mathbb{Z}^2 \times \{0,1, \ldots, k-1\}$ and edge set $\mathcal{E}_k = \{\{x,y\} : \|x-y\|_1 = 1\}$. As is usual, we take an initial spin configuration $\sigma(0) = (\sigma_x(0))_{x \in S_k}$ on $\Omega_k = \{-1,1\}^{S_k}$ distributed using the product measure of $\mu_p,~ p \in [0,1]$, where
\[
\mu_p(\sigma_x(0)=+1) = p = 1- \mu_p(\sigma_x(0)=-1)\ .
\]

The configuration $\sigma(t)$ evolves as $t$ increases according to the zero-temperature limit of Glauber dynamics (the majority rule). To describe this, define the energy (or local cost function) of a site $x$ at time $t$ as 
\[
e_x(t) = - \sum_{y : \{x,y\} \in \mathcal{E}_k} \sigma_x(t)\sigma_y(t)\ .
\]
Note that up to a linear transformation, this is just the number of neighbors $y$ of $x$ such that $\sigma_y(t) \neq \sigma_x(t)$. Each site has an exponential clock with different clocks independent of each other and when a site's clock rings, it makes an update according to the rules
\[
\sigma_x(t) = \begin{cases}
- \sigma_x(t^-) & \text{ if } e_x(t^-) > 0 \\
\pm 1 & \text{ with probability }1/2 \text{ if } e_x(t^-) = 0 \\
\sigma_x(t^-) & \text{ if } e_x(t^-) < 0
\end{cases}\ .
\]
Write $\mathbb{P}_p$ for the joint distribution of $(\sigma(0),\omega)$, the initial spins and the dynamics realizations.

The main questions we will address involve fixation. We say that the slab $S_k$ fixates for some value of $p$ if 
\[
\mathbb{P}_p(\text{there exists }T=T(\sigma(0),\omega) <\infty \text{ such that } \sigma_0(t) = \sigma_0(T) \text{ for all } t \geq T) = 1\ .
\]
We will actually only focus on the case $p=1/2$, so write $\mathbb{P}$ for $\mathbb{P}_{1/2}$. The setup thus far corresponds to the model with free boundary conditions; in the case of periodic boundary conditions, we consider sites of the form $(x,y,k-1)$ and $(x,y,0)$ to be neighbors in $S_k$. If $k=2$ then this enforces two edges between $(x,y,1)$ and $(x,y,0)$, so that in the computation of energy of a site, that neighbor counts twice.

\section{Main results}
The first theorem concerns fixation for small $k$. We will prove the case $k=2$ in the next section; the case $k=3$ will be treated in a companion paper \cite{DKNS13}. That paper will also contain a simplified proof of the case $k=2$, notable for removing the bootstrap percolation comparison used here. 
\begin{thm}\label{thm: k=2}
For $k=2$ with free or periodic boundary conditions, $S_k$ fixates. For $k=3$ with periodic boundary conditions, $S_k$ fixates.
\end{thm}

The proof of the following theorem is in Section~\ref{sec: blinkers}. The construction used in the proof for $k=4$ with periodic boundary conditions is considerably more involved and will not be given in this paper; it will appear in \cite{DKNS13}.
\begin{thm}\label{thm: kabove2}
With $k \geq 4$ and periodic boundary conditions, $S_k$ does not fixate. With $k\geq 3$ and free boundary conditions, $S_k$ does not fixate. 
\end{thm}

\section{Proof of Theorem~\ref{thm: k=2} for $k=2$}

For the free boundary condition case, the theorem follows from the argument in Nanda-Newman-Stein \cite{NNS00}. Specifically, for $v,v' \in S_2$, define $m_t(v',v)$ as the contribution to $e_v(t) - e_v(0)$ due to flips of the spin $\sigma_{v'}$. Write $\pi:S_2 \to \mathbb{Z}^2$ for the projection $\pi(x,y,z) = (x,y)$ and for $v \in S_2$, we use the notation that $\hat v$ is the vertex in $S_2$ with $\hat v \neq v$ but $\pi(\hat v) = \pi(v)$. Then
\[
\mathbb{E} \left[ e_v(t) - e_v(0) \right] = \sum_{v'\in S_2 : \|v-v'\|_1 \leq 1} \mathbb{E} m_t(v',v) = \mathbb{E} m_t(v,v) + \mathbb{E} m_t(\hat v, v)+ \sum_{\stackrel{v' \in S_2 : \|v-v'\|_1=1}{v' \neq \hat v}} \mathbb{E} m_t(v',v)\ .
\]
By symmetry, $\mathbb{E} m_t(v', v) = \mathbb{E} m_t(v,v')$ for all $v'$ so this equals
\[
\mathbb{E} m_t(v,v) + \mathbb{E} m_t(v,\hat v) + \sum_{\stackrel{v' \in S_2 : \|v-v'\|_1=1}{v' \neq \hat v}} \mathbb{E} m_t(v,v')\ .
\]
Note that whenever $v$ flips, the sum of the changes of $e_{v'}(t)-e_{v'}(0)$ for all neighbors $v'$ is simply equal to the change of $e_v(t)-e_v(0)$. Therefore
\[
\mathbb{E} \left[ e_v(t) - e_v(0) \right] = 2\mathbb{E} m_t(v,v)\ .
\] 
Because $v$ has 5 neighbors, $e_v(t)$ decreases by at least 2 each time $\sigma_v$ flips, so $\mathbb{E} m_t(v,v)$ is bounded above by $-2\mathbb{E} N_t(v)$, where $N_t(v)$ is the number of flips of $\sigma_v$ until time $t$. Taking $t$ to infinity and noting that $|e_v(t) - e_v(0)| \leq 10$ for all $t$, we see that almost surely, $\sigma_v$ flips finitely often.

For the periodic case, we will use the following fact several times. If $A \subset \Omega_2$ then we say that $\sigma(t) \in A$ \emph{infinitely often} if the set $\{t: \sigma(t) \in A\}$ is unbounded. To avoid technical issues, we will restrict our attention to $A$'s that are cylinder sets.
\begin{lma}\label{lma: stopping_time}
If $A$ and $B$ are (cylinder) events in $\Omega_2$ such that 
\[
\inf_{\sigma \in A} \mathbb{P}(\sigma(t) \in B \text{ for some } t \in (0,1] \mid \sigma(0)=\sigma) > 0\ ,
\]
then 
\[
\mathbb{P}(\sigma(t) \in A \text{ infinitely often but } \sigma(t) \in B \text{ finitely often})=0\ .
\]
\end{lma}
\begin{proof}
The proof is just an application of the strong Markov property at a sequence of stopping times $(\mathcal{T}_k)$, which could be given by $\mathcal{T}_0=0$ and
\[
\mathcal{T}_k = \inf \{t \geq \mathcal{T}_{k-1}+2 : \sigma(t) \in A\}\ .
\]
\end{proof}

Let us say that a site $v \in S_2$ \emph{fixates} for the realization $(\sigma(0),\omega)$ if there exists $T_v = T_v(\sigma(0),\omega) < \infty$ such that $\sigma_v(t) = \sigma_v(T_v)$ for all $t \geq T_v$. We say that $v$ \emph{fixates from time $T$} if for all $t \geq T$, $\sigma_v(t) = \sigma_v(T)$.

We now define a process $\tau(t) = (\tau_y(t): y \in \mathbb{Z}^2$) from $\sigma(t)$ by declaring $\tau_{\pi(v)}(t)=\sigma_v(t)$ if $\sigma_v(t)=\sigma_{\hat v}(t)$. Otherwise we declare $\tau_{\pi(v)}(t)$ to be grey. In the latter case, we refer to $\pi(v)$ as a $(+/-)$ or $(-/+)$ site if the site of $v,\hat v$ with third coordinate 1 is $+1$ or $-1$, respectively. We will use the terms `flip' and `fixate' for the configuration $\tau(t)$ as well. Note that with probability one, a site cannot flip from grey to grey; that is, it cannot flip from $(+/-)$ to $(-/+)$ or $(-/+)$ to $(+/-)$. We may say that $\pi(v)$ fixates at $+$; this means that $\pi(v)$ fixates and that its terminal value is $+$. We define $\pi(v)$ fixating at $-$ or at grey (either $(+/-)$ or $(-/+)$) similarly.

\begin{lma}\label{lma: grey}
With probability one, no site in $\mathbb{Z}^2$ can fixate at grey. 
%That is, with probability one, for all $y \in \mathbb{Z}^2$ and all $t\geq 0$, if $\tau_y(t)$ is grey, then we cannot have $\tau_y(T)=\tau_y(t)$ for all $T \geq t$.
\end{lma}
\begin{proof}
Let $v \in S_2$ and $A_v \subset \Omega_2$ be the event that $\sigma_v=+1$ but $\sigma_w=-1$ for at least 3 neighbors of $v$ (counting $\hat v$ twice). Let $B_v$ be the event but that $\sigma_v=-1$ but $\sigma_w=-1$ for at least 3 neighbors of $v$. There is some $c>0$ such that
\[
\mathbb{P}\left( \sigma_v(t) \in B_v \text{ for some } t \in (0,1] ~\bigg|~ \sigma(0) = \sigma \right) \geq c \text{ for all } \sigma \in A_v\ .
\]
For instance, $v$'s clock may ring before those of all its neighbors and $\sigma_v$ then flips. Using Lemma~\ref{lma: stopping_time},
\[
\mathbb{P}\left( \sigma(t) \in A_v \text{ infinitely often but } \sigma(t) \text{ flips only finitely often}\right) = 0\ .
\]

Suppose that for some $y \in \mathbb{Z}^2$ and $t\geq 0$, $\tau_y(t)$ is grey but $\tau_y(T)=\tau_y(t)$ for $T \geq t$. Write $v$ for the site in $S_2$ with $\pi(v)=y$ and third coordinate equal to 1. Assume without loss in generality that $\tau_y$ is ($+/-$). Note that then $v$ already has at least two unsatisfied neighbors (since we count $\hat v$ twice). Therefore off of the probability zero event above, all neighbors of $v$ with third coordinate equal to 1 must fixate in $\sigma(t)$ at $+1$. Similarly, all neighbors of $\hat v$ with third coordinate equal to 0 fixate in $\sigma(t)$ at $-1$. Iterating this argument, with positive probability, the top level of $S_2$ fixates in $\sigma(t)$ at $+1$ and the bottom fixates at $-1$. By ergodicity under spatial translations, this event would have probability 1 but this contradicts symmetry under permuting the top and bottom levels.
\end{proof}

\begin{lma}\label{lma: fixate}
With probability one, for all $y \in \mathbb{Z}^2$, if $y-(1,0)$ and $y+(0,1)$ fixate in $\tau(t)$, then so does $y$.
\end{lma}
\begin{proof}
%By symmetry, it suffices to consider two cases: that both $y-(1,0)$ and $y+(0,1)$ fixate and that both $y-(1,0)$ and $y+(1,0)$ fixate. 
Suppose that $y-(1,0)$ and $y+(0,1)$ fixate in $\tau(t)$. If they both fixate at $+$ then the argument is not difficult: if $y$ does not fixate it must be grey (say $(+/-)$) infinitely often. But in this case, we can apply Lemma~\ref{lma: stopping_time} to the event $A$ that $\sigma_{y-(1,0)}$ and $\sigma_{y+(0,1)}$ are $+1$ with $\tau_y$ equal to $(+/-)$ and $B$ the event that $\sigma_{y-(1,0)}$ and $\sigma_{y+(0,1)}$ are $+1$ with $\tau_y$ equal to $+$. This proves that $\tau_y$ cannot be grey infinitely often without being $+$ infinitely often. But because two neighbors of $y$ have $\tau$-value fixed at $+$, $\tau_y$ will then remain at $+$ after some time.

Otherwise $y-(1,0)$ and $y+(0,1)$ fixate at different $\tau$-values, say $+1$ and $-1$ respectively. We will use the following fact: with probability one, each spin $\sigma_v$ can have only finitely many energy-lowering flips. In other words, for each $v \in S_2$ and $t\geq 0$ we can define $F_v(t)$ to be the number of times $s \in (0,t)$ such that $\sigma_v(s^-) \neq \sigma_v(s^+)$ and $e_v(s^+)<e_v(s^-)$. Since the measure $\mathbb{P}$ is invariant under translations, the argument of Newman-Nanda-Stein \cite{NNS00} can be applied to find
\[
\lim_{t \to \infty} F_v(t) <\infty \text{ with probability one}\ .
\]
As a consequence of this and Lemma~\ref{lma: stopping_time}, we see that for each $v \in S_2$, 
\begin{equation}\label{eq: energy_lowering}
\mathbb{P}\left( \sigma_v(t) \text{ disagrees with at least } 4 \text{ neighbors of }v \text{ in }S_2 \text{ infinitely often} \right) = 0\ .
\end{equation}

Assume that the $\tau$-value at $y$ does not fixate; then it must be grey (for example $(+/-)$) infinitely often. Note that at each of these times, each $\sigma_v$ spin at a site $v$ with $\pi(v)=y$ disagrees with at least 3 neighbors. From the above remarks, there must be some random time at which these spins no longer disagree with at least 4 neighbors. This implies that infinitely often 
\[
\tau_{y-(1,0)}=+,~\tau_y=+/-,~\tau_{y+(0,1)}=- \text{ and } ~\tau_{y+(1,0)}= \tau_{y-(0,1)}=+/-\ .
\]

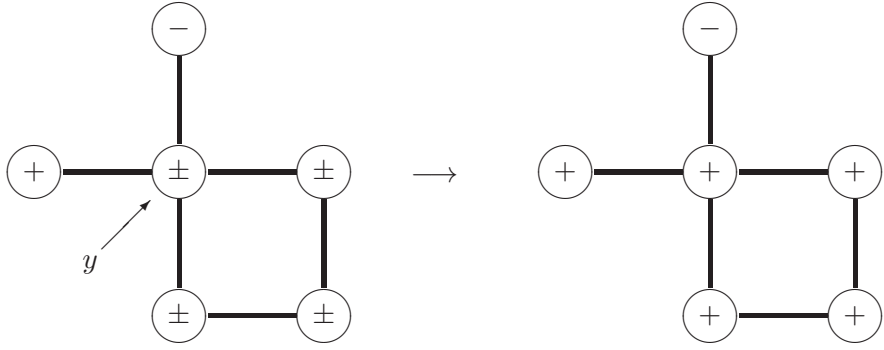
\begin{figure}
\setlength{\unitlength}{.5in}
\begin{picture}(10,3)(-5.5,2.5)
\put(-3.5,5){\circle{.5}}
\put(-2,5){\circle{.5}}
\put(-.5,5){\circle{.5}}
\put(-2,6.5){\circle{.5}}
\put(-2,3.5){\circle{.5}}
\put(-.5,3.5){\circle{.5}}
\put(-3.62,4.93){$+$}
\put(-2.12,4.93){$\pm$}
\put(-.62,4.93){$\pm$}
\put(-2.12,6.42){$-$}
\put(-2.12,3.43){$\pm$}
\put(-.62,3.43){$\pm$}
\linethickness{1.5pt}
\put(-1.7,3.5){\line(1,0){.91}}
\put(-1.7,5){\line(1,0){.91}}
\put(-3.2,5){\line(1,0){.91}}
\put(-2,3.8){\line(0,1){.91}}
\put(-2,5.3){\line(0,1){.91}}
\put(-.5,3.8){\line(0,1){.91}}
\put(-3,4){$y$}
\put(-2.8,4.2){\vector(1,1){.5}}
\put(.4,4.9){$\longrightarrow$}
\put(2,5){\circle{.5}}
\put(3.5,5){\circle{.5}}
\put(5,5){\circle{.5}}
\put(3.5,6.5){\circle{.5}}
\put(3.5,3.5){\circle{.5}}
\put(5,3.5){\circle{.5}}
\put(1.88,4.93){$+$}
\put(3.38,4.93){$+$}
\put(4.88,4.93){$+$}
\put(3.38,6.42){$-$}
\put(3.38,3.43){$+$}
\put(4.88,3.43){$+$}
\linethickness{1.5pt}
\put(3.8,3.5){\line(1,0){.91}}
\put(3.8,5){\line(1,0){.91}}
\put(2.3,5){\line(1,0){.91}}
\put(3.5,3.8){\line(0,1){.91}}
\put(3.5,5.3){\line(0,1){.91}}
\put(5,3.8){\line(0,1){.91}}
%\put(2.5,4){$y$}
%\put(2.7,4.2){\vector(1,1){.5}}
\end{picture}
\caption{The local configuration $\tau$ near $y$ on the left, in which $\tau_y = (+/-), \tau_{y-(1,0)} = +, \tau_{y+(0,1)} = -$ and $\tau_y=(+/-)$. The $\tau$-values at $y+(1,0)$ and $y-(0,1)$ are $(+/-)$. In the case depicted, we also have $\tau_{y+(1,-1)} = (+/-)$. One example of a finite sequence of flips that can occur is as follows. $\tau_y$ flips to $+$, $\tau_{y+(1,0)}$ flips to $+$, $\tau_{y-(0,1)}$ flips to $+$ and then $\tau_{y+(1,-1)}$ flips to $+$. This eventually fixates $\tau$-values as on the right.}
\label{fig: fig_1}
\end{figure}

We now consider the $\tau$-value of $y+(1,-1)$ at these times $\mathcal{T}$. There must exist one status from the choices $+,-,(+/-)$ and $(-/+)$ such that this spin has this status infinitely often (of the times $\mathcal{T}$). But now, it is elementary (though a bit tedious) to see that in each case, there is a finite sequence of flips that will lead all eight $\sigma_v$'s for $v \in S_2$ with $\pi(v)$ in the set $\{y,y+(1,0),y-(0,1),y+(1,-1)\}$ to have the same sign -- see Figure~\ref{fig: fig_1} for an example. Using Lemma~\ref{lma: stopping_time} completes the proof, because once they are the same sign, they can never flip again.
\end{proof}

To complete the proof we invoke a comparison to bootstrap percolation, giving a version of van Enter's argument \cite{VE87} initially due to Straley. For any $\sigma \in \Omega_2$ we identify a configuration $\eta = \eta(\omega) \in \{0,1\}^{\mathbb{Z}^2}$ as follows. We declare $\eta_x = 1$ if all $v$ in the $2\times 2 \times 2$ block $B_x = 2x+\{0,1\}^3$ have spins of the same sign in $\sigma$. Note that under the coarsening dynamics, all such sites are fixated in $S_2$. For all other sites $x$ we set $\eta_x=0$. We then run the following discrete time (deterministic) dynamics on $\eta$. We set $\eta(0) = \eta(\omega)$ with $\omega$ distributed by $\mathbb{P}_{1/2}$ and for each $n \in \mathbb{N}$ and $x \in \mathbb{Z}^2$, we set $\eta_x(n) = 1$ if either (a) $\eta_x(n-1)=1$ or (b) $\eta_{y_1}(n-1) = \eta_{y_2}(n-1) = 1$ for at least two neighbors $y_i$ of $x$ with $\|y_1-y_2\|_\infty = 1$. Otherwise we set $\eta_x(n)=0$. This is a modified bootstrap percolation dynamics.

We claim that with probability one, for each $x$, the value $\eta_x(n)$ is $1$ for all large $n$. Using Lemma~\ref{lma: fixate}, this will prove that all sites in $S_2$ fixate. To show the claim we briefly summarize the classic argument of \cite{VE87}. Because the $\eta_x(0)$ variables are independent from site to site, one can show that for some $n$, the probability is positive that all sites in the rectangle $[0,n]^2$ begin with $\eta$-value $1$ but that there is no rectangular contour enclosing $[0,n]^2$ all of whose sites begin with $\eta$-value $0$. On this event, under our dynamics, such a rectangle will eat away at all of space and fix all sites to have $\eta$-value $1$. However, by the ergodic theorem, with probability one, some translate of this event will occur and this completes the proof.

\section{Proof of Theorem~\ref{thm: kabove2}}\label{sec: blinkers}

We begin by proving the case $k=3$ with free boundary conditions. The idea is to force a large rectangle on level 3 (that is, with third coordinate equal to $2$) to be fixed at $+1$ with a parallel region on level 1 (third coordinate equal to $0$) fixed at $-1$. Spins on the middle level between these regions act like spins in the coarsening model on $\mathbb{Z}^2$.

\begin{figure}
\setlength{\unitlength}{.5in}
\begin{picture}(10,3)(-5.5,2.5)
\multiput(-1.5,3)(.75,0){7}{\line(0,1){3.75}}
\multiput(-1.5,3)(0,.75){6}{\line(1,0){4.5}}
\multiput(-1.25,3.3)(.75,0){6}{$+$}
\multiput(-1.25,4.05)(.75,0){6}{$+$}
\multiput(-1.25,6.3)(.75,0){6}{$-$}
\multiput(-1.25,5.55)(.75,0){6}{$-$}
\multiput(-1.25,4.8)(.75,0){2}{$+$}
\multiput(1.75,4.8)(.75,0){2}{$-$}
\multiput(1.75,5.55)(.75,0){2}{$-$}
\multiput(-1.13,3.37)(0,.75){5}{\circle{.5}}
\multiput(-.38,3.37)(0,.75){5}{\circle{.5}}
\multiput(1.87,3.37)(0,.75){5}{\circle{.5}}
\multiput(2.62,3.37)(0,.75){5}{\circle{.5}}
\end{picture}
\caption{Level 1 (in $\mathbb{Z}^2 \times \{1\}$) in the event $A$, for a slab of width 3 with free boundary conditions. The left unmarked box represents the vertex $(0,0,1)$. The vertices with circled spins are ones both of whose third coordinate neighbors (``above'' and ``below'') have the same spin. Any configuration in $A$ has the property that spins at vertices above those in the uncircled region are $-1$ and below those are $+1$. The unmarked spins flip infinitely often.}
\label{fig: fig_2}
\end{figure}

To stabilize levels 1 and 3, we define for $m, n \in \mathbb{Z}$ the set $P_{m,n} = \{m,m+1\}\times \{n,n+1\} \times \{0,1,2\}$ and the ``table'' of size $n\geq 2$
\[
T_n = \bigg[ \{-n, \ldots, n\}^2 \times \{2\} \bigg] \cup P_{-n,-n} \cup P_{-n,n-1} \cup P_{n-1,-n} \cup P_{n-1,n-1}\ .
\]
The inverted table of size $n$, $T_n'$, is the reflection of $T_n$ through $\mathbb{Z}^2 \times \{1\}$. Note that if either of these sets are initially monochromatic, then they will be fixed by the coarsening dynamics. Define the event $A \subset \Omega_3$ that
\begin{enumerate}
\item all sites in $P_1=P_{-2,-2}\cup P_{-2,-1} \cup P_{2,-2}$ have spin $+1$ and all sites in $P_2=P_{-2,1} \cup P_{2,0} \cup P_{2,1}$ have spin $-1$,
\item all sites in $\{0,1\}\times \{-2,-1\} \times \{1\}$ have spin $+1$ and all sites in $\{0,1\}\times \{1,2\}\times \{1\}$ have spin $-1$,
\item all sites in $T_{10}' \setminus \bigg[ P_1 \cup P_2 \bigg]$ have spin $+1$ and
\item all sites in $T_{20} \setminus \bigg[ T_{10}' \cup P_1 \cup P_2 \bigg]$ have spin $-1$.
\end{enumerate}
The reader may verify that all sites in $T_{10} \cup T_{20}' \cup P_1 \cup P_2$ are fixated in the event $A$. However, the vertex $(0,0,1)$ then has 3 plus neighbors so by Lemma~\ref{lma: stopping_time} it must have a plus spin infinitely often. This implies that the vertex $(1,0,1)$ has 3 plus neighbors infinitely often and therefore must have a plus spin infinitely often. By symmetry, the same is true for these vertices and minus spin, meaning they flip infinitely often. As usual, by spatial ergodicity, almost surely some translate of this event occurs and therefore with probability one, not all sites fixate.

The cases $k \geq 3$ with free boundary conditions are handled similarly. We simply add more layers of the construction on top of level 2. To define the event precisely, we set $A$ to be the event defined exactly as in the case of $k=3$ (above). This event only involves the first three levels $(0-2)$ of the slab. Define $A' = \{\sigma\}$ as the event that $\sigma \in A$ and that for all $(x,y,k)$ with $(x,y) \in \{-20, \ldots, 20\}^2$ and $k \geq 3$, we have $\sigma_{(x,y,k)} = \sigma_{(x,y,2)}$. Because in the slab $S_3$, the event $A$ forced all spins for vertices in $\{-20, \ldots, 20\}^2 \times \{2\}$ to be fixed, it is not hard to check that on $A'$, all spins for vertices in $\{-20, \ldots, 20\}^2 \times \{2, \ldots, k-1\}$ are also fixed. The same argument as before gives that the spins at $(0,0,1)$ and $(1,0,1)$ do not fixate and consequently the slab does not fixate.

For the case $k \geq 5$ with periodic boundary conditions, we consider again the event $A$ and add duplicate layers of the third level as before. The only difference is that we also need to add a duplicate layer of the zeroth level at the top (which is the same as level $-1$), in the set $\mathbb{Z}^2 \times \{k-1\}$. Because we need to duplicate both level 0 and 2, this requires at least 5 layers. The proof is now complete.

\medskip
\noindent
{\bf Acknowledgements.} M. D. thanks C. M. N. and the Courant Institute for support.


\thebibliography{1}

\bibitem{A83}
R. Arratia. (1983). Site recurrence for annihilating random walks on $\mathbb{Z}^d$. \emph{Ann. Probab.} {\bf 11} 706--713.

\bibitem{DKNS13}
M. Damron, H. Kogan, C. M. Newman and V. Sidoravicius. (2013). In preparation.

\bibitem{VE87}
A. C. D. van Enter. (1987). Proof of Straley's argument for bootstrap percolation. \emph{J. Statist. Phys.} {\bf 48} 943--945.

\bibitem{KRB10}
P. L. Krapivsky, S. Redner and E. Ben-Naim. A kinetic view of statistical physics. \emph{Cambridge University Press, Cambridge,} 2010. 504 pp. 

\bibitem{NNS00}
S. Nanda, C. M. Newman and D. L. Stein. Dynamics of Ising spin systems at zero temperature. In \emph{On Dobrushin's Way (from Probability Theory to Statistical Mechanics)}, R. Minlos, S. Shlosman and Y. Suhov, eds., Amer. Math. Soc. Transl. (2) {\bf 198}:183--193 (2000).

\bibitem{NS00}
C. M. Newman and D. L. Stein. (2000). Zero-temperature dynamics of Ising spin systems following a deep quench: results and open problems. \emph{Physica A}. {\bf 279} 159--168.

\bibitem{S94}
D. Stauffer. (1994). Ising spinodal decomposition at $T=0$ in one to five dimensions. \emph{J. Phys. A.} {\bf 27} 5029--5032.
\end{document}